\documentclass[11pt]{article}
\usepackage{mathrsfs}
\usepackage{latexsym,lineno}
\usepackage{epsfig}
\usepackage{color}
\usepackage{indentfirst}
\usepackage{amsmath}\usepackage{fleqn}\usepackage{verbatim}\usepackage{epsf}
\usepackage{amsthm}\usepackage{graphicx, float}\usepackage{graphicx}
\usepackage{amsfonts}\usepackage{amssymb}\usepackage{graphpap}
\usepackage{epic}\usepackage{curves}
\newcommand{\be}{\begin{equation}}
\newcommand{\ee}{\end{equation}}
\newcommand{\benum}{\begin{enumerate}}
\newcommand{\eenum}{\end{enumerate}}
\newcommand{\bit}{\begin{itemize}}
\newcommand{\eit}{\end{itemize}}
\newtheorem{thom}{Theorem}[section]
\newtheorem{lemma}{Lemma}[section]

\newtheorem{prop}{Proposition}[section]

\newtheorem{claim}{Claim}

\usepackage{graphicx}

\begin{document}
\def\s{\subseteq}
\def\n{\noindent}
\def\se{\setminus}
\def\dia{\diamondsuit}
\def\la{\langle}
\def\ra{\rangle}

%--------------------------------------------------------------

\title{On the maximum and minimum multiplicative Zagreb indices of graphs with  given number of cut edges}

\author{Shaohui Wang$^{a}$, ~Chunxiang Wang$^{b,}$\footnote{  Authors' email addresses:    S. Wang (shaohuiwang@yahoo.com),  C. Wang (wcxiang@mail.ccnu.edu.cn),  L. Chen  (409137414@qq.com).}, ~Lin Chen$^b$   \\\small\emph {a. Department of Mathematics and Computer Science, Adelphi University, Garden City, NY 11550, USA} \\
\small\emph {b. School of Mathematics and Statistics, Central China Normal University, Wuhan,
430079, P.R. China} }
\date{}
\maketitle

\begin{abstract}

For a  molecular  graph, the first multiplicative Zagreb index $\Pi_1$ is equal to the product of the square of the degree of the vertices, while the second multiplicative Zagreb index $\Pi_2$ is equal to the product of the endvertex degree of each edge over all edges. Denote by $\mathbb{G}_{n,k}$ the set of graphs with $n$ vertices and $k$ cut edges. In this paper, we explore   graphs in terms of    a number of cut edges. In addition, the maximum and minimum multiplicative Zagreb indices of graphs with   given number of cut edges are provided.  Furthermore, we characterize  graphs with the largest and smallest $\Pi_1(G)$ and $\Pi_2(G)$ in $\mathbb{G}_{n,k}$, and our results extend and enrich some known conclusions.

\end{abstract}

\section{Introduction}

In the interdisciplinary of mathematics, chemistry and physics,
 molecular invariants/descriptors can be applied the study of quantitative structure-property relationships (QSPR) and
quantitative structure-activity relationships (QSAR),  and
for the descriptive purposes of biological and chemical
properties, such as 
   boiling and melting points and toxicity~\cite{Gutman1996}. One type of the most classical topological molecular descriptors
is named as Zagreb indices $M_1$ and $M_2$~\cite{Gutman1972},  which are literal
quantities in an expected formula for the total
$\pi$-electron energy of conjugated molecules.
In the view of successful considerations on
the  applications on Zagreb indices~\cite{Gutman2014},
Todeschini et al.(2010)~\cite{RT20101,RT20102,Wang2015} introduced the
 multiplicative variants of molecular structure
descriptors, denoted by $\Pi_1$ and $\Pi_2$ the multiplicative Zagreb indices.
(Multiplicative) Zagreb indices are
employed as molecular descriptors in QSPR and QSAR, see \cite{10,11}.

 Mathematicians have exhibited considerable interest
in the properties of Zagreb indices about the extremal values or bounds for the topological indices of graphs,
as well as related problems of characterizing the extremal graphs~\cite{Hu2005,Li2008,shi2015,BF2014,SM2014,Xu2014,WangJ2015,Liu2015,LiuP2015,LiuPX2015,Wang2017}. In addition to a plenty of applications for the Zagreb indices in chemistry, there are many
situations in 
multiplicative Zagreb indices, which  attracted one of the focus of interests in  physics and graph theory. Borovi\'canin et al.~\cite{Borov2016} investigated upper
bounds on Zagreb indices of trees in terms of domination number
and extremal trees are characterized.
Wang and Wei~\cite{Wang2015}  introduced sharp upper and lower bounds of these indices in
$k$-trees.
Liu and Zhang [14] provided several sharp upper bounds for
$\pi_1$-index and $\pi_2$-index in terms of graph parameters
such as the order, size and radius~\cite{Liuz20102}. Wang et al. \cite{LiuPX2015} obtained extremal multiplicative Zagreb indices of trees with given
number of vertices of maximum degree.
Xu and Hua~\cite{Xu20102}  explored an unified approach to characterize extremal (maximal and minimal) trees, unicyclic graphs and bicyclic graphs with respect to multiplicative Zagreb
indices, respectively. 
Iranmanesh et al.~\cite{Iranmanesh20102}  gave the first and the second multiplicative Zagreb indices for a class of chemical moleculor of dendrimers. 
 Also, a lower bound for the
first Zagreb index of trees with a given domination number is
determined and the extremal trees are characterized as well.
  Kazemi~\cite{Ramin2016} studied the bounds for
the moments and the probability generating function of these
indices in a randomly chosen molecular graph with tree structure
of order $n$.  The connected graphs with $k$ cut edges (or vertices) have been considered in many mathematical literatures~\cite{Liu2004,Wu2007,Zhao2010,Deng2010,Bojana2015}. It is natural to consider that, for the $n$-vertex tree, $k=n-1$, and trees with the extremal  multiplicative  Zagreb indices had been studied a long time ago \cite{Iranmanesh20102}.

In view of the  above results, in this paper we further
investigate multiplicative Zagreb indices of graphs with a given number of cut edges. In addition, the maximum and minimum of $\Pi_1(G) $ and $\Pi_2(G) $ of graphs with   given number of cut edges are provided.  Furthermore, we characterize  graphs with the largest and smallest multiplicative Zagreb indices in $\mathbb{G}_{n,k}$.

\section{Preliminary}

  Denote  by $G = (V, E)$  a simple connected graph with $n$ vertices and $m$ edges, where $V = V (G)$ is called vertex set and $E = E(G)$ is called edge set. For  $v\in V(G)$, $N(v)$ denotes the neighbors of $v$, that is,  $N_{G}(v)= \left\{ u|\; uv\in E(G)\right\}$, and $d_{G}(v)= \left| N(v)\right|$ is the degree of $v$. The first and second Zagreb indices \cite{Gutman2014} of a graph $G$ are given by
\begin{eqnarray} \nonumber
M_1(G) = \sum_{u \in V(G)} d(u)^2
~\text{ and}
~M_2(G) = \sum_{uv \in E(G)} d(u)d(v).
\end{eqnarray}
The first multiplicative Zagreb index $\Pi_1 =\Pi_1(G)$ and the second  multiplicative  Zagreb index $\Pi_2 =\Pi_2(G) $ \cite{RT20101,RT20102} of a graph $G$ are defined as
\begin{eqnarray} \nonumber
 \Pi_1(G) = \prod_{u \in V(G)} d(u)^2 ~
\text{ and}\;\;
 \Pi_2(G) = \prod_{uv \in E(G)} d(u)d(v) = \prod_{u \in V(G)}
d(u)^{d(u)}.
\end{eqnarray}

A pendent vertex is a vertex of degree one   and a supporting vertex is a vertex adjacent to at least one pendent vertex. A pendent edge is incident to a pendent vertex and  a supporting vertex. 
For  two graphs $G_1$ and $G_2$, if there exists a common vertex $v$ between them, then let  $G_1vG_2$ be  a graph such that the vertex set of  $G_1vG_2$ is $V(G_1)\bigcup V(G_2)$, $V(G_1)\bigcap V(G_2) ={v}$  and $E(G_1vG_2)= E(G_1)\bigcup E(G_2)$. If   $G_1$, $G_2$, $\cdots$, $G_l$ with $l \geq 2$   share a common vertex $v$, then by $G_1vG_2v\cdots vG_l$ denote this graph.  For $u_1 \in V(G_1)$ and $u_s \in V(G_2)$,  if $P= u_1u_2\cdots u_s$ is a path, then denote this graph by $G_1PG_2$ or $G_1u_1u_2\cdots u_sG_2$ in which $P$ is called an internal path. For $S \subseteq V(G)$ and  $F \subseteq E(G)$,   we use $G[S]$ for the subgraph of $G$ induced by $S$, $G - S$ for the subgraph induced by $V(G) - S$ and $G - F$ for the subgraph of G obtained by deleting $F$. A vertex $u$ (or an edge $e$, respectively) is called a cut vetex (or cut edge, respectively) of a connected graph $G$, if $G- v$ (or $G -e$) has at least two components.   A graph G is said to be $2$-connected  if there does not exist a vertex whose removal disconnects the graph. A block is a connected graph which does not have any cut vertex, and $K_2$ is a trivial block. The endblock contains at most one cut vertex.
As usual, $P_n$, $S_n$ and $C_n$ are a path, a star and a cycle on $n$ vertices, respectively. The cyclomatic number of a connected graph $G$ is given by $c(G)= m-n+1$. In particular, if $c(G)= 0,\,1$ and $2$, then $G$  is a tree, unicyclic graph and bicyclic graph, respectively. If a connected graph on $n$ vertices has the cyclomatic number at least one, then the number of its cut edges is at most $n-3$. Thus, we assume that $G$ has $1 \leq k \leq n-3$ cut edges   in our following discussion.

Let $\mathbb{G}_{n,k}$ be a set of graphs with $n$ vertices and $1 \leq k \leq n-3$ cut edges, and $E_c= \left\{e_1, e_2, \cdots, e_k\right\}$ be a set of cut edges of $G$. Then $E_c$ can be considered as two categories, which are the pendent edges and non-pendent edges (or internal paths of length $1$). 
The components of $G - E_c$ are $2$-connected graphs and isolated vertices.   Denote by
$K_n^S$ (or $K_n^P$, respectively)  a graph obtained by identifying (connecting to, respectively) the nonpendent vertex of a star $S_k$  (or a pendent vertex of a path $P_k$, respectively) to a vertex of $K_{n-k}$  (see Fig 1). In addition, let $C_n^S$ (or $C_n^P$, respectively) be a graph obtained by identifying (connecting to, respectively) the nonpendent vertex of a star $S_k$  (or a pendent vertex of a path $P_k$, respectively) to a vertex of $C_{n-k}$.
 \begin{figure}[htbp]
    \centering
    \includegraphics[width=5in]{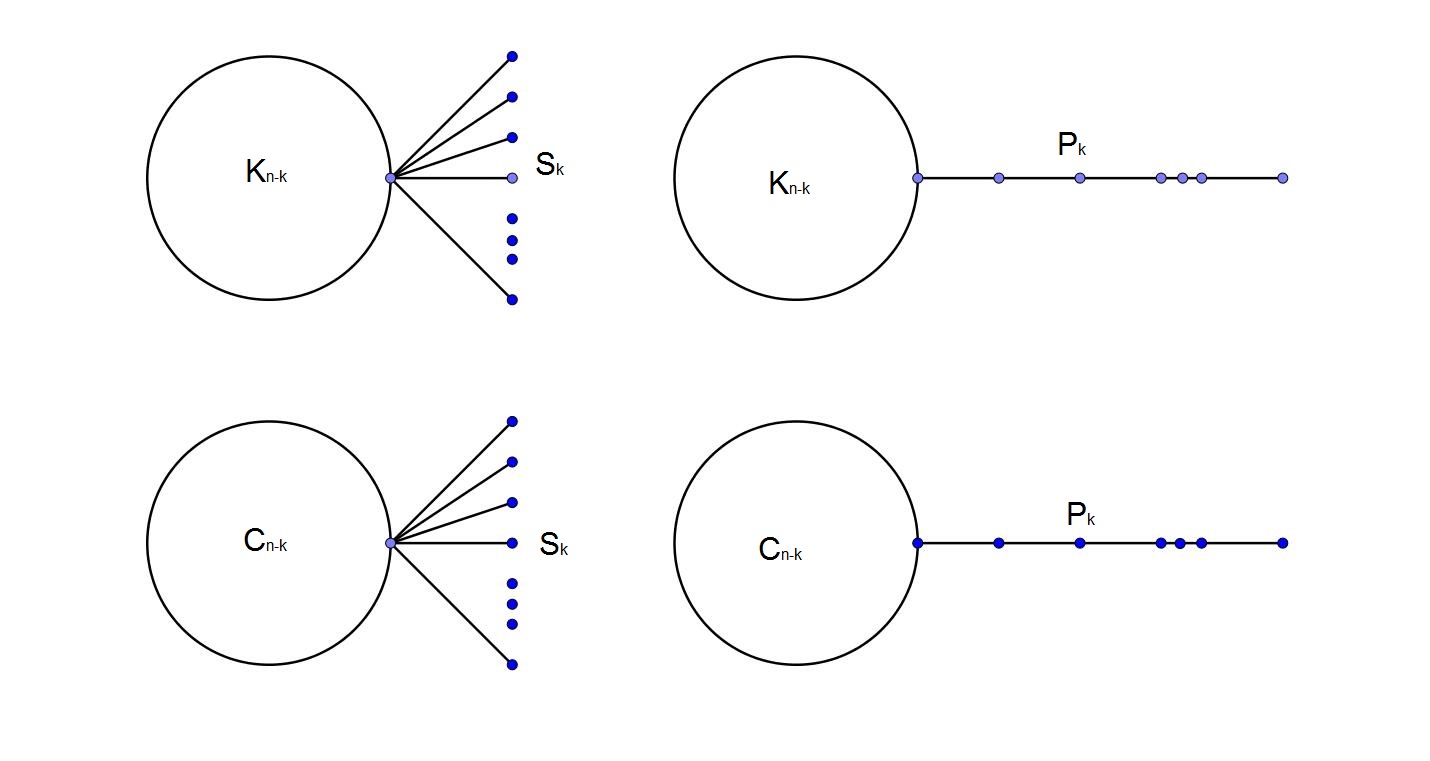}
   \caption{ $K_n^S, K_n^P, C_n^S$ and $C_n^P$.}
   \label{fig: te}
\end{figure}

In our exposition we will use the terminology and
notations of (chemical) graph theory(see \cite{BB,NT}). By elementary calculations, one can derive the following propositions.

\begin{prop}\label{p1}
Let $t(x) = \frac{x}{x+m}$ be a function with $m > 0$. Then $t(x)$ is increasing in $\mathbb{R}$.
\end{prop}

\begin{prop}\label{p2}
Let $l(x) = \frac{x^x}{(x+m)^{x+m}}$ be a function with $m > 0$. Then $l(x)$ is decreasing  in $\mathbb{R}$.
\end{prop}

Based on the concepts of $\Pi_1(G)$ and $\Pi_2(G)$, we have
\begin{lemma}\label{l1}
Let $G=(V, E)$ is a graph and $i=1, \,2$.  \\ 
  (i) If $e=uv\notin E(G)$, $u,v\in V(G)$, then $ \Pi_i(G+uv) > \Pi_i(G)$.\\
  (ii) If $e=uv\in E(G)$, then $\Pi_i(G-e) < \Pi_i(G)$.
\end{lemma}

Lemma 2.1 yields the following result.
\begin{lemma}\label{l2}
Let $G=(V, E)$ be a  $2$-connected graph and $i = 1, 2$. \\
(i) If~$\Pi_i(G)$ is  maximal, then $G$   is the complete graph $K_n$.  \\
(ii) If~$\Pi_i(G)$ is  minimal, then $G$ is the cycle $C_n$.
 \end{lemma}

\begin{lemma}\label{l7}
Let $C_1, C_2$ be cycles, and $P_s= u_1u_2 \cdots u_s$ be an internal path of $G = C_1P_sC_2$ such that $u_1 \in V(C_1)$ and $u_s \in V(C_2)$. Assume that $u_1v_1, u_1v_2 \in  E(C_1)$ and $u_sw_1, u_sw_2 \in E(C_2)$.  Let $G' = G- \{u_1v_2, u_sw_1, u_sw_2 \} + \{v_2w_2, u_1w_1\}$. Then $\Pi_i(G) > \Pi_i(G')$ with $i=1, 2$.
\end{lemma}
\begin{proof}
By the transformation from $G$ to $G'$, we have $d_{G'}(u_s) =1 < d_G(u_s) = 3$. For $v \in V(G)- \{u_s\}$, $d_G(v) = d_{G'}(v)$. Then $\Pi_i(G) > \Pi_i(G')$ with $i=1, 2$, and we complete the proof.
\end{proof}

\begin{lemma} \label{l3}
Let $G_1P_mG_2$ and $G_1G_2P_m$ be graphs   (see Fig 2), in which  $P_m$ is a   path, and  $G_1$, $G_2$ are connected. Then $\Pi_1(G_1P_mG_2) \geq \Pi_1(G_1G_2P_m)$ and  $\Pi_2(G_1P_mG_2) \leq \Pi_2(G_1G_2P_m)$.
\end{lemma}

 \begin{figure}[htbp]
    \centering
    \includegraphics[width=6in]{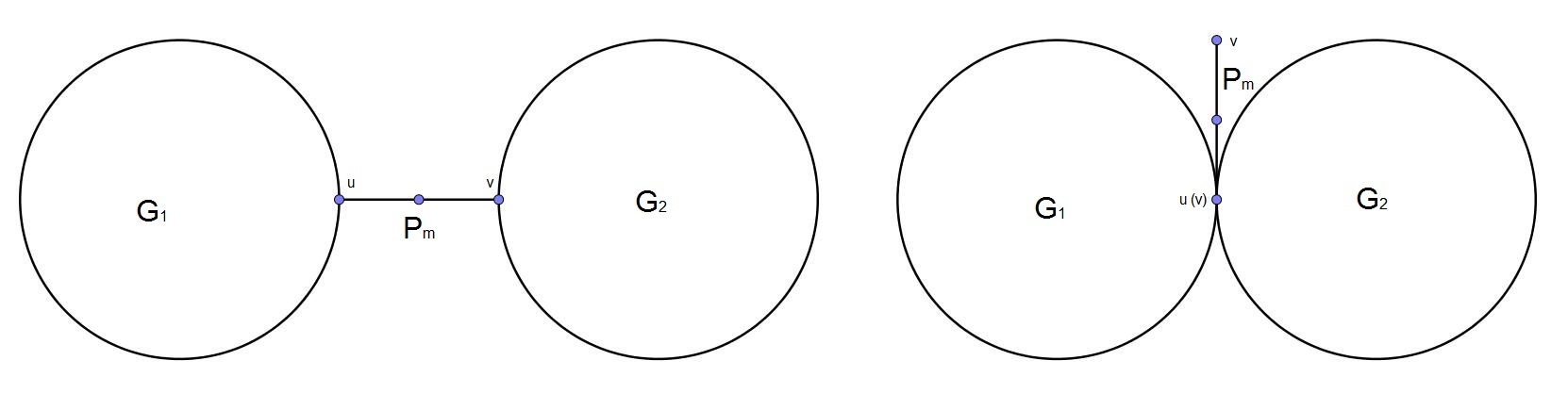}
 \caption{ $G_1P_mG_2$ and $G_1G_2P_m$.}
   \label{fig: te}
\end{figure}
\begin{proof}   Let $d_{G_1P_mG_2}(u)=x$ and $d_{G_1P_mG_2}(v)=y$. Then $d_{G_1G_2P_m}(u)=x+y-1$. By the definitions of multiplicative Zagreb indices, we obtain
\begin{eqnarray}
 \frac{\Pi_1(G_1P_mG_2)}{\Pi_1(G_1G_2P_m)}
= \frac{x^2y^2}{(x+y-1)^2 1^2}
                         = \bigg(\frac{\frac{x}{x+y-1}  }{\frac{1}{1+(y-1)}} \bigg)^{2}. \nonumber
\end{eqnarray}
Since $x\geq 1, y \geq 1$, and by Propostion \ref{p1}, we have $\Pi_1(G_1P_mG_2) \geq \Pi_1(G_1G_2P_m)$.
Note that
\begin{eqnarray}
 \frac{\Pi_2(G_1P_mG_2)}{\Pi_2(G_1G_2P_m)}= \frac{x^{x}y^{y}}{(x+y-1)^{(x+y-1)} 1^{1}}
                         =\frac{\frac{x^{x}}{{(x+y-1)}^{(x+y-1)}}}{\frac{1^{1}}{{(1+y-1)}^{(1+y-1)}}}.\nonumber
\end{eqnarray}
By $x\geq 1$ and  Proposition \ref{p2}, we have  $\frac{\Pi_2(G_1P_mG_2)}{\Pi_2(G_1G_2P_m)} \leq 1$, that is, $\Pi_2(G_1P_mG_2) \leq \Pi_2(G_1G_2P_m)$. Thus, this completes the proof.
\end{proof}

From Lemma \ref{l3}, we have the useful lemma below.
\begin{lemma}\label{l4}
 Let $GT$ be a graph by indentifying a vertex of a tree $T \ncong S_n$ to a vertex $u$ of $G$, and $GS$ be a graph by attaching $|E(T)|$ pendent edges to $u$   (see Fig 3).
 Then $\Pi_1(GT)> \Pi_1(GS)$ and $\Pi_2(GT) < \Pi_2(GS)$.
 \end{lemma}
\begin{figure}[htbp]
\centering
\includegraphics[width=6in]{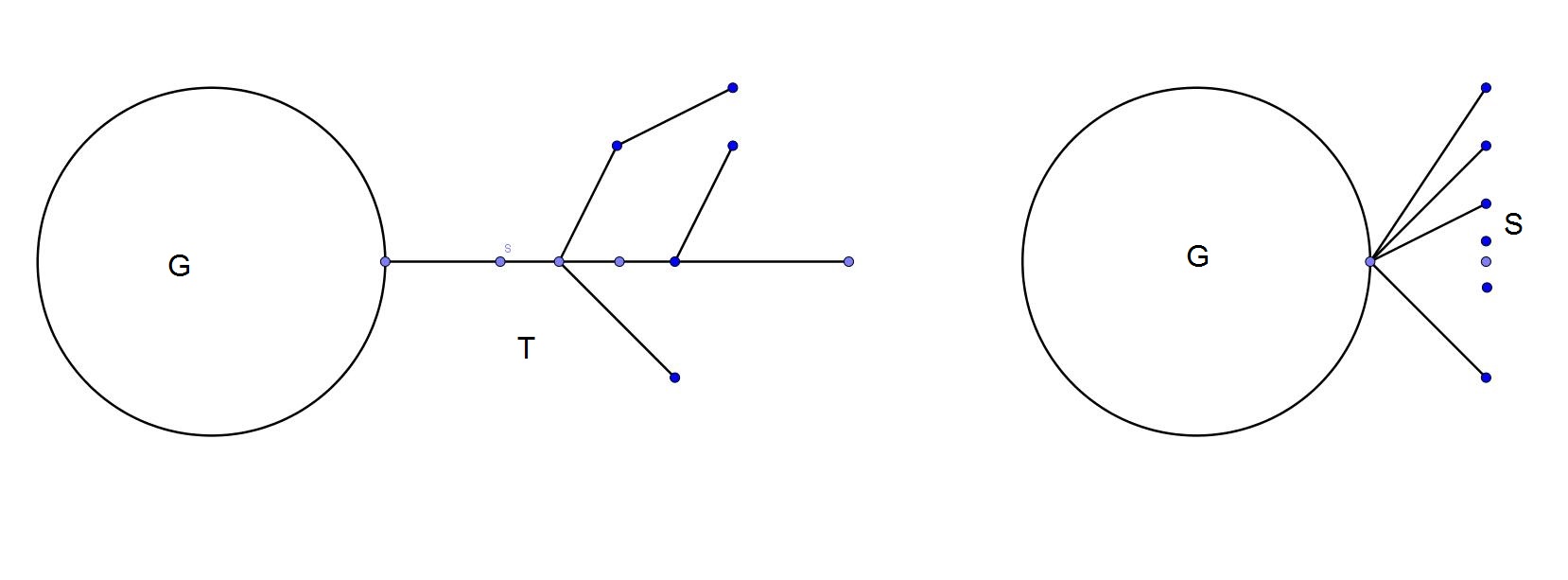}
 \caption{ $GT$ and $GS$.}
\label{fig: te}
\end{figure}

% Based on the above proof, the following remark is immediate.
%{\bf  Remark 3.1. } By  Proposition 2.4, we adding edges to the vertices of the 2-edge-connected subgraph $S$ with $|S|\geq 3$, thus, $S$ can be changed into the completed subgraph,   Repeating the grafting transformation $A$ or $B$, any non-pendant path can changed into pendant path.

\begin{lemma}\label{l5} 
Let $u$ ($v$, respectively) be a vertex in $G$, and $u_1, u_2, \dots, u_s$ be the endvertices of pendent path $P_1, P_2, \cdots , P_s$ ($v_1, v_2, \dots, v_t$ be the endvertices of  $P_1', P_2', \cdots , P_t'$, respectively). Set
$uu_i' \in E(P_i)$ with $1 \leq i \leq s$, and $vv_j' \in E(P_j')$ with $1 \leq j \leq t$.  Let $G'=G-\left\{uu_i'\right\}+\left\{vu_i'\right\}$ with $1 \leq i \leq s$, $G''=G-\{vv_j'\}+\{uv_j'\}$ with $1 \leq j \leq t$ and $\left|V(G_0)\right|\geq 3$   (see Fig 4). 
 Then either $\Pi_1(G) \geq \Pi_1(G')$ and $\Pi_2(G) \leq \Pi_2(G')$,   or $\Pi_1(G) > \Pi_1(G^{''})$ and $\Pi_2(G) < \Pi_2(G^{''})$.
\end{lemma}
\begin{figure}[htbp]
    \centering
    \includegraphics[width=6in]{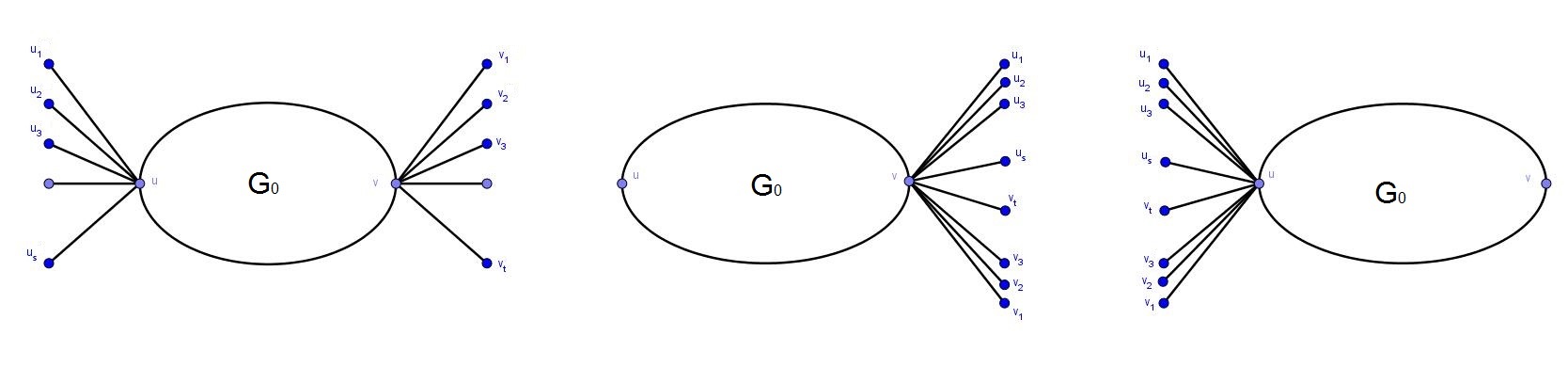}
\caption{ $G$, $G'$ and $G''$.}
   \label{fig: te}
\end{figure}

\begin{proof}Let $d_G(u)=x$, $d_G(v)=y$. By the constructions of $G'$ and $G''$ , we have $d_{G'}(u)=d_G(u)-s=x-s$,  $d_{G'}(v)=d_G(v)+s=y+s$,
$d_{G''}(u)=d_G(u)+t=x+t$ and   $d_{G''}(v)=d_G(v)-t=y-t.$
Combining with the concepts of mutiplicative Zagreb indices, we have 
 \begin{eqnarray}
   \frac{\Pi_1(G)}{\Pi_1(G')}=\frac{x^{2}y^{2}}{(x-s)^{2}(y+s)^{2}}
                                 =\frac{(\frac{y}{y+s})^{2}}{(\frac{x-s}{(x-s)+s})^{2}},\nonumber
 \end{eqnarray}
 \begin{eqnarray}
   \frac{\Pi_2(G)}{\Pi_2(G')}=\frac{x^{x}y^{y}}{(x-s)^{x-s}(y+s)^{y+s}}
                                 =\frac{\frac{y^{y}}{(y+s)^{y+s}}}{\frac{(x-s)^{x-s}}{x^{x}}}
                                 =\frac{\frac{y^{y}}{(y+s)^{y+s}}}{\frac{(x-s)^{x-s}}{{[(x-s)+s]}^{(x-s)+s}}},\nonumber
 \end{eqnarray}
 \begin{eqnarray}
   \frac{\Pi_1(G)}{\Pi_1(G'')}=\frac{x^{2}y^{2}}{(x+t)^{2}(y-t)^{2}}
                                  =\frac{(\frac{x}{x+t})^{2}}{(\frac{y-t}{(y-t)+t})^{2}} \nonumber
 \end{eqnarray}
and
 \begin{eqnarray}
   \frac{\Pi_2(G)}{\Pi_2(G'')}=\frac{x^{x}y^{y}}{(x+t)^{x+t}(y-t)^{y-t}}
                                  =\frac{\frac{x^{x}}{(x+t)^{x+t}}}{\frac{(y-t)^{y-t}}{y^{y}}}
                                  =\frac{\frac{x^{x}}{(x+t)^{x+t}}}{\frac{(y-t)^{y-t}}{[(y-t)+t]^{(y-t)+t}}}.\nonumber
 \end{eqnarray}

If $y\geq x-s$, by Propositions \ref{p1} and \ref{p2}, we can obtain that $\Pi_1(G) \geq \Pi_1(G')$ and $\Pi_2(G) \leq \Pi_2(G')$.
If $y \leq x-s -1$,  then   $x \geq y+s+1 > y-t$.  Propositions \ref{p1} and \ref{p2} yield that  $\Pi_1(G) > \Pi_1(G'')$ and $\Pi_2(G) < \Pi_2(G'')$. Thus, the lemma  is proved.
\end{proof}

%{\bf  Remark 3.2. } Repeating grafting transformation $C$, all the pendant paths are attached to the same vertex.

\begin{lemma}\label{l6}
Let $P_1=u_1u_2 \cdots u_s$ and $P_2 = v_1v_2 \cdots v_t$ be two pendent path of $G$ with $s, t \geq 2$ and $d(u_s) = d(v_t) =1$   (see Fig 5). Let $G' = G - v_1v_2 + u_sv_2$. Then $\Pi_1(G) < \Pi_1(G')$  and $\Pi_2(G) > \Pi_2(G')$.
\end{lemma}
\begin{figure}[htbp]
    \centering
    \includegraphics[width=5in]{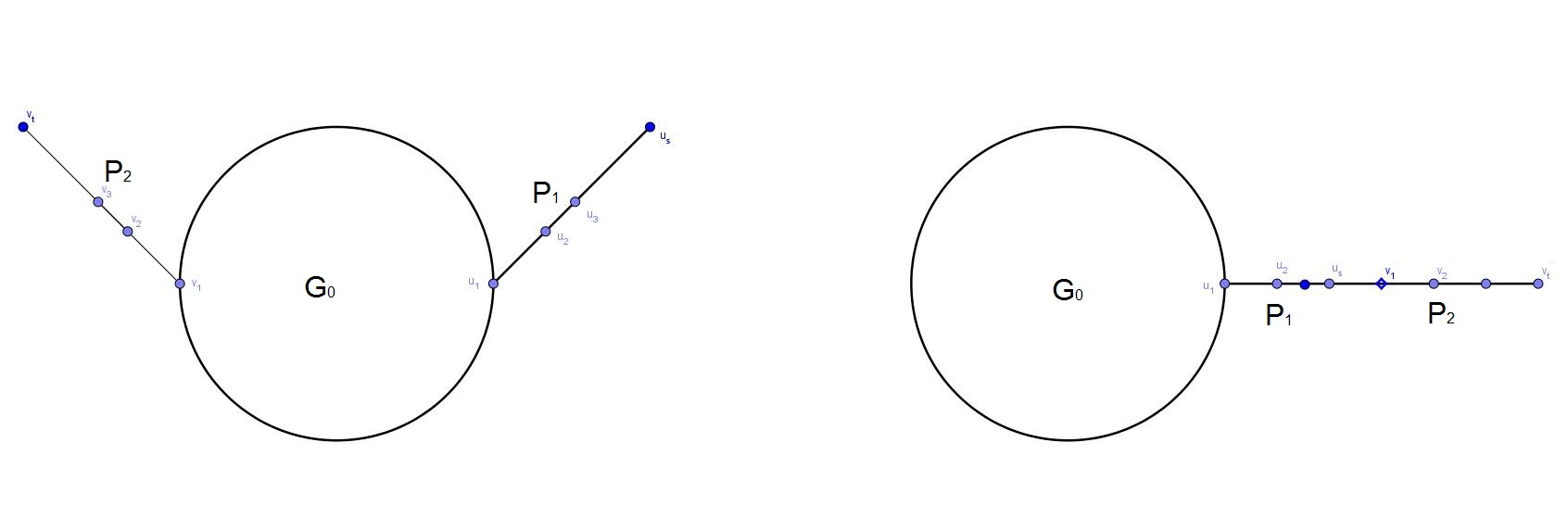}
\caption{ $G$ and $G'$.}
   \label{fig: te}
\end{figure}
\begin{proof} Note that $d(u_1) \geq 3, d(v_1) \geq 3$.
By the definitions of mutiplicative Zagreb indices, we have
\begin{eqnarray}
   \frac{\Pi_1(G)}{\Pi_1(G')}=\frac{d(u_s)^2 d(v_1)^2}{d_{G'}(u_s)^2 d_{G'}(v_1)^2}
                                  =\bigg(\frac{\frac{1}{2}}{\frac{d(v_1)-1}{d(v_1)}}\bigg)^2.\nonumber
 \end{eqnarray}
By Proposition \ref{p1}, we have $\frac{\Pi_1(G)}{\Pi_1(G')} < 1$, that is, $\Pi_1(G) < \Pi_1(G')$. 
 \begin{eqnarray}
   \frac{\Pi_2(G)}{\Pi_2(G')}=\frac{d(u_s)^{d(u_s)} d(v_1)^{d(v_1)}}{d_{G'}(u_s)^{d_{G'}(u_s)} d_{G'}(v_1)^{d_{G'}(v_1)}}
                                  =\bigg(\frac{\frac{~1^1~}{~2^2~}}{\frac{(d(v_1)-1)^{d(v_1)-1}}{d(v_1)^{d(v_1)}}}\bigg)^2. \nonumber
 \end{eqnarray}
By Proposition \ref{p2}, we have $\frac{\Pi_2(G)}{\Pi_2(G')} > 1$, that is, $\Pi_2(G) > \Pi_2(G')$.  

Thus, this completes the proof.
\end{proof}

\section{ Graphs with smallest  multiplicative Zagreb indices  in $\mathbb{G}_{n,k}$}

In this section, we  discuss graphs with the smallest $\Pi_1(G)$ and $\Pi_2(G)$  in $\mathbb{G}_{n,k}$. 
\begin{thom}\label{t3}
Let $G$ be a graph in $\mathbb{G}_{n,k}$ with $1 \leq k \leq n-3$.  Then
$$\Pi_1(G) \geq  4^{n-k-1}(k+2)^2,$$
where the  equality holds if and only $G \cong  C_n^S$, respectively.
\end{thom}
\begin{proof}
Choose a graph $G \in \mathbb{G}_{n,k}$ such that $\Pi_1(G)$ is as small as possible.
Let $E_c$ be a   cut  edge set of $G$ and $B_1, B_2, \cdots, B_{k+1}$ be the components of $G- E_c$. By Lemma \ref{l2}, we have $B_i$ is a cycle or an isolated vertex. Lemma \ref{l7} implies that $G$ has a unique cycle.  By Lemma \ref{l4}, all cut edges in $G$ are pendent edge. By Lemma \ref{l5}, all pendent edges share a common supporting vertex, that is, $G \cong C_n^S$. Thus, this completes the proof.
\end{proof}

\begin{thom}\label{t4}
Let $G$ be a graph in $\mathbb{G}_{n,k}$ with $1 \leq k \leq n-3$.  Then 
$$\Pi_2(G) \geq 27*4^{n-2},$$ 
where the equality holds if and only $G \cong C_n^P$.
\end{thom}
\begin{proof}
Let $G \in \mathbb{G}_{n,k}$ be a graph  such that $\Pi_2(G)$ is minimal.
Let $E_c$ be a  cut   edge set of $G$ and $B_1, B_2, \cdots, B_{k+1}$ be the components of $G- E_c$. By Lemma \ref{l2}, we have $B_i$ is a cycle  or an isolated vertex. Lemma \ref{l7} implies that $G$ has a unique cycle. By Lemma \ref{l6}, there is only one pendent path in $G$.
Thus $G \cong C_n^P$, and we prove this theorem.
\end{proof}

\section{ Graphs with   largest multiplicative Zagreb indices  in $\mathbb{G}_{n,k}$ }

We proceed to consider graphs with the largest $\Pi_1(G)$ and $\Pi_2(G)$  in $\mathbb{G}_{n,k}$ in this section. 
\begin{thom}\label{t1}
Let $G$ be a graph in $\mathbb{G}_{n,k}$ with $1 \leq k \leq n-3$.  Then 
$$\Pi_1(G) \leq 4^{k-1}(n-k)^2(n-k-1)^{2(n-k-1)},$$
where the equality holds if and only $G \cong K_n^P$.
\end{thom}
\begin{proof}
Denote by a graph $G \in \mathbb{G}_{n,k}$  such that $\Pi_1(G)$ is maximal.
Set $E_c$ to be a   cut edge set of $G$ and $B_1, B_2, \cdots, B_{k+1}$ the components of $G- E_c$. By Lemma \ref{l2}, we have $B_i$ is a clique of size at least 3 or an isolated vertex.  Next we start with the follwing claims.

\begin{claim}\label{c4}
Every two cliques of size at least 3 do not share a common vertex.
\end{claim}
\noindent{\small  \em Proof of Claim \ref{c4}.}
We proceed it by a contradiction. Assume there are at least two blocks $B_1, B_2$ shared a common vertex $v_0$ in $G$ such that $|B_1|, |B_2| \geq 3$. Choose $v_1 \in V(B_1)$,  $v_2 \in V(B_2)$ and $v_1, v_2 \neq v_0$. Let $G' = G+ v_1v_2$. By Lemma \ref{l1}, $\Pi_2(G') > \Pi_2(G)$, that is a contradiction to the assumption of $G$.  The claim is proved.

We introduce a graph transformation that is used in the rest of our proof.

\begin{claim}\label{c8} 
Let $K_{n_1}$ and $K_{n_2}$ be two farthest endblocks of $K_{n_1}G_0K_{n_2}$ such that $v_{11} \in V(K_{n1}) \cap V(G_0)$ and $v_{l1} \in V(K_{n2}) \cap V(G_0)$   (see Fig 6). If $d(v_{11}) = n_{1} \geq 3$ and $d(v_{l1}) = n_2 \geq 3$, then $\Pi_1(K_{n_1}G_0K_{n_2}) < \Pi_1(K_{n_1+n_2-1}G_0).$
\end{claim}
\begin{figure}[htbp]
    \centering
    \includegraphics[width=7in]{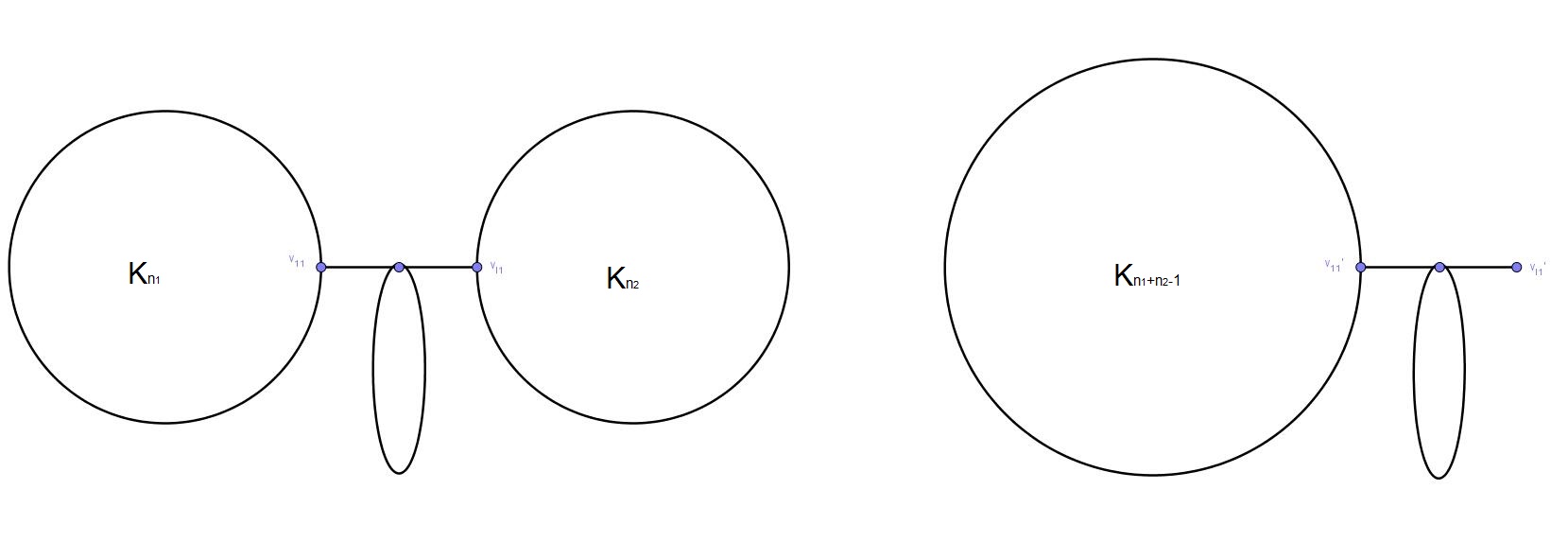}
\caption{ $G$ and $G'$.}
   \label{fig: te}
\end{figure}

\noindent{\small  \em Proof of Claim \ref{c8}.}
Let $V(K_{n_1}) = \{ v_{11}, v_{12}, \cdots, v_{1n_1}\}$ and  $V(K_{n_2}) = \{ v_{l1}, v_{l2}, \cdots, v_{ln_2}\}$.  Denote  by $G = K_{n_1}G_0 K_{n_2}$ and $G' = G- \{v_{l1}v_{li}, i \geq 2\} + \{v_{li}v_{1j}, i \geq 2, j \geq 1\}=K_{n_1+n_2 -1}G_0$.
By the concepts of multiplicative Zagreb indices, we have 
 \begin{eqnarray}
\frac{\Pi_1(G)}{\Pi_1(G')} &=& \bigg(\frac{d(v_{11})d(v_{12})d(v_{13}) \cdots d(v_{1n_1}) d(v_{l1})d(v_{l2})d(v_{l3}) \cdots d(v_{ln_2})}{d'(v_{11})d'(v_{12})d'(v_{13}) \cdots d'(v_{1n_1}) d'(v_{l1})d'(v_{l2})d'(v_{l3}) \cdots d'(v_{ln_2})}\bigg)^2 \nonumber \\
&=& \bigg(\frac{n_1n_2(n_1-1)^{n_1-1}(n_2-1)^{n_2-1}}{(n_1+n_2-1) (n_1+n_2-2)^{n_1+n_2-2}} \bigg)^2\nonumber\\
 &\leq& 
\bigg(\frac{n_1n_2(n_1-1)^{n_1-1}(n_2-1)^{n_2-1}}{ (n_1+n_2-2)^{n_1+n_2-1}}\bigg)^2.\nonumber
 \end{eqnarray}
Let $f(x) = \frac{xn_2(x-1)^{x-1} (n_2-1)^{n_2-1}}{(x+n_2-2)^{x+n_2-1}}$. Then we take a derivative of $ln({f(x)})$ as $\frac{1}{x} + ln(x-1) +1 - ln(x+n_2-2) - \frac{x+n_2-1}{x+n_2-2} < \frac{1}{x} + ln(x-1) - ln(x+n_2-2) \leq \frac{1}{x} + ln(x-1)- ln(x+1)$, by $n_2 \geq 3$.

Set $g(x) = \frac{1}{x} + ln(x-1)- ln(x+1)$. Note that $g'(x) = \frac{x^2+1}{x^2(x^2-1)} > 0$ and $lim_{x \rightarrow  \infty} g(x) =  lim_{x \rightarrow  \infty} ln(\frac{(x-1)e^{\frac{1}{x}}}{x+1} ) = 0$, by L' Hospital's Rule. Thus, $g(x) < 0$.
So, $f(x)$ is a decreasing function  and $$\frac{\Pi_1(G_1)}{\Pi_1(G_2)} \leq \frac{3n_2 (3-1)^{3-1} (n_2 -1)^{n_2-1}}{(3+n_2-2)^{3+n_2-1}} = \frac{12*n_2*(n_2-1)^{n_2-1}}{(n_2+1)^2 (n_2+1) (n_2+1)^{n_2-1}}.$$ 
Since $12 \leq (n_2+1)^2$ and  $n_2 < n_2+1$, then $\frac{\Pi_1(G_1)}{\Pi_1(G_2)} < 1.$ 
This completes the proof of Claim \ref{c8}

\begin{claim}\label{c5}
There exsits exactly one path.
\end{claim}
\noindent{\small  \em Proof of Claim \ref{c5}.}
We prove it by  contradictions. Assume  that there are at least two paths $P_1=u_1u_2 \cdots u_s, P_2 = v_1v_2 \cdots v_l$ with $d(u_1), d(v_1) \geq 3$. We consider three cases that $P_i$  is either a pendent path or an internal path with $i = 1, 2$.

\noindent{\bf \small Case 1.}
$d(u_s) = d(v_l) = 1.$

\noindent{\small  \em Proof of Case 1.}
By Lemma \ref{l6},   there is another graph $G' \in \mathbb{G}_n^k$ such that $\Pi_1(G) < \Pi_1(G')$, which is a contradiction to the choice of $G$.
% Let $G' = G- u_{1}u_2 + v_l u_2$. Note that $\frac{\Pi_1(G)}{\Pi_1(G')} = \frac{d(u_1)^2 d(u_2)^2 d(v_l)^2}{d_{G'} (u_1)^2 d_{G'} (u_2)^2 d_{G'} (v_l)^2 } = \big(\frac{d(u_1) *2 * 1}{ (d(u_1)-1) *2 * 2}\big)^2 = \big(\frac{\frac{1}{2} }{ \frac{d(u_1) -1}{ d(u_1)}}\big)$. Since $d(u_1) -1 >1$, 

\noindent{\bf \small Case 2.}
$d(u_s) =1,  d(v_l) \geq 3.$

\noindent{\small  \em Proof of Case 2.} Let $G'' = G- \{v_1v_2, u_1u_2\} + \{v_1u_2, v_2u_s\}$. Note that  $$\frac{\Pi_1(G)}{\Pi_1(G'')} = \frac{d(u_1)^2 d(u_s)^2}{ d_{G''}(u_1)^2 d_{G''}(u_s)^2} = \bigg(\frac{ \frac{~1~}{2}}{\frac{d(u_1) -1}{ d(u_1)}}\bigg)^2.$$
Since $d(u_1) \geq 3$, by Proposition \ref{p1}, we have $\Pi_1(G) < \Pi_1(G'')$, that is a contradiction to the choice of $G$.

\noindent{\bf \small Case 3.}
$d(u_s) \geq 3,  d(v_l) \geq 3.$

\noindent{\small  \em Proof of Case 3.} 
By Case 2, there does not exist any pendent paths in $G$. Then every cut edge is  in an internal path. By choosing two farthest endblocks and Claim 2, there is another graph $G'''$ such that $\Pi_1(G''') > \Pi_1(G)$, which contradicted that $\Pi_1(G)$ is maximal. This completes the proof of Case 3.

%We first consider that all internal paths are length 1. By Claim 2, there exists another graph $G_c$ such that $\Pi(G_c) > \Pi(G)$, which contradicted that $\Pi_1(G)$ is maximal.

% By Proposition \ref{p1}, $\Pi_1$ is maximal if and only if  there are $k+1$ cliques of the same size of %$\frac{n}{k+1}$ and every two closest cliques are connected by an edge.
%So this graph has 2k vertices of degree $\frac{n}{k+1}$ and $n-2k$ vertices of degree $\frac{n}
%{k+1}-1$. By elementary calculations and the definition of first multiplicative Zagreb index, we have %$\Pi_1(G) < \Pi_1(K_n^P)$, which contradicted that $\Pi_1(G)$ is maximal.
 
%Secondly, we assume that there are at least one internal path of length at least two.  Without loss of generality, $s \geq 2$,  $l \geq 3$ and  $u_s$ is closer to $v_2$ than $u_1$. Then $d(v_2) = d(v_3) = d(v_{l-1}) =2$. Let  $G''' = G - \{u_{s-1}u_s\} + \{u_{s-1}v_2\}$.
% Then $G'''$ is connected and
%$$\frac{\Pi_1(G)}{\Pi_1(G''')} = \frac{d(u_s)^2 d(v_2)^2}{(d(u_s)-1)^2 (d(v_2)+1)^2} = \bigg(\frac{\frac{~2~}{3} }{ \frac{d(u_s)-1}{d(u_s)}}\bigg)^2.$$ 
%By Proposition \ref{p1} and 
%$d(u_s) -1 \geq 2 $, we have
%$\Pi_1(G) \leq  \Pi_1(G''')$.  We continue to use  Claim \ref{c8},  then there is another graph $G'''' \in \mathbb{G}_n^k$ such that $\Pi_1(G''') \leq \Pi(G'''')$. This is a  contradition to the choice of $G$. 
%This completes the proof of Case 3.

Therefore,  $G$ contains a unique clique of size at least $3$ and the unique path is a pendent path.
Thus $G \cong K_n^P$, and this completes the proof.
\end{proof}

\begin{thom}\label{t2}
Let $G$ be a graph in $\mathbb{G}_{n,k}$  with $1 \leq k \leq n-3$.  Then 
$$\Pi_2(G) \leq (n-1)^{n-1} (n-k-1)^{(n-k-1)^2},$$
where the equality holds if and only $G \cong K_n^S$.
\end{thom}
\begin{proof}
Pick a graph $G \in \mathbb{G}_{n,k}$ such that $\Pi_2(G)$ is as large as possible.
Denote by $E_c$   a   cut edge set of $G$ and $B_1, B_2, \cdots, B_{k+1}$ be the components of $G- E_c$. By Lemma \ref{l2}, we have $B_i$ is a clique of size at least 3 or an isolated vertex. By Lemma \ref{l3}, if  two blocks  are connected by a path, then they share  a common vertex.  

\begin{claim}\label{c1}
There is a uniqe block $B$ such that $|B| \geq 3$.
\end{claim}
\noindent{\small  \em Proof of Claim \ref{c1}.}
We proceed it by a contradiction. Assume that there are at least two blocks $B_1, B_2$ shared a common vertex $v_0$ in $G$ such that $|B_1|, |B_2| \geq 3$. Choose $v_1 \in V(B_1)$ and $v_2 \in V(B_2)$ and $v_1, v_2 \neq v_0$. Let $G' = G+ v_1v_2$. By Lemma \ref{l1}, $\Pi_2(G') > \Pi_2(G)$ and this claim is proved.

 By Lemmas \ref{l4} and \ref{l5}, we have $G \cong K_n^S$, and this completes the proof.
\end{proof}

\vskip4mm\noindent{\bf Acknowledgements}

%\textcolor{blue}{The authors would like to express their sincere
%gratitude to the three
% anonymous referees and the editor for many friendly
%and helpful suggestions, which led to great deal of improvement of
%the original manuscript.}

 The work is partially supported by the National Natural Science Foundation of China under Grants 11371162, 11571134, 
 and the Self-determined Research Funds of CCNU from the colleges basic research and operation of MOE.

\end{document}